\documentclass[11pt]{article}
\usepackage[T1]{fontenc}
\usepackage{lmodern}
\usepackage[a4paper,margin=27mm]{geometry}
\usepackage{amsmath,amssymb,amsthm,mathtools}
\usepackage{booktabs}
\usepackage{enumitem}
\usepackage{microtype}
\usepackage[hidelinks]{hyperref}
\numberwithin{equation}{section}
\newtheorem{theorem}{Theorem}[section]
\newtheorem{proposition}[theorem]{Proposition}
\newtheorem{lemma}[theorem]{Lemma}
\newtheorem{corollary}[theorem]{Corollary}
\DeclareMathOperator{\conv}{conv}
\newcommand{\R}{\mathbb R}

\newcommand{\one}{\mathbf 1}
\newcommand{\PP}{\mathcal P}
\newcommand{\Acal}{\mathcal A}
\newcommand{\tr}{\mathsf T}
\newcommand{\normmax}[1]{\lVert #1\rVert_{\max}}
\newcommand{\normone}[1]{\lVert #1\rVert_1}
\title{A Uniform Bound on Optimal Strategy Length in Water Transport Problem}
\author{Tianyi Tao\textsuperscript{1,*} \and Bohan Yang\textsuperscript{2}}
\date{}
\begin{document}
\raggedbottom
\maketitle
\begin{NoHyper}
\begingroup
\renewcommand{\thefootnote}{}
\footnotetext{\textsuperscript{*}Corresponding author. Email: tytao@scu.edu.cn (T. Tao); bhyang@simis.cn (B. Yang).}
\addtocounter{footnote}{-1}
\endgroup
\end{NoHyper}
\begin{abstract}
We prove that every water transport problem on an $n$-vertex graph has an
optimal strategy of length at most $n^{(2+o(1))n}$. More strongly, the convex
hull of all strategy operators stabilizes within the same bound. We also give
a five-vertex instance in which every optimal strategy repeats a nontrivial
connected averaging set.
\end{abstract}

\section{Introduction}

Let $G=(V,E)$ be a finite simple graph, let
$\omega:V\to\R_{\ge0}$ be an initial weight function, and fix a target
vertex $v_\ast\in V$. An \emph{averaging operation} chooses a nonempty
connected set $S\subseteq V$ and replaces the weights on $S$ by their common
average. A strategy is a finite sequence of averaging operations, possibly empty. The
water transport problem asks for a strategy maximizing the final weight at
$v_\ast$. Throughout the paper, strategy length counts complete averaging
operations on connected vertex sets, as defined above. We make no assertion
about finite optimal strategy lengths in the edge-only sharing model.

For a nonempty connected set $S$, let $A_S$ denote the corresponding
averaging operator. If $\sigma=(S_1,\ldots,S_k)$ is listed in chronological
order, write
\[
 A_\sigma=A_{S_k}\cdots A_{S_1},\qquad |\sigma|=k.
\]
Let $t(G,v_\ast,\omega)$ be the minimum length of an optimal strategy, with
$t(G,v_\ast,\omega)=+\infty$ if the supremum is not attained.

The pairwise Sharing a Drink procedure was introduced by
H\"aggstr\"om~\cite{Haggstrom2012}; H\"aggstr\"om and Hirscher later
studied the water transport problem~\cite{HH2015}. Gollin et
al.~\cite[Corollary~4.7]{GHH+2025} proved that every instance admits a finite
optimal strategy. Vilkas obtained another finite-attainment result using
complete hypermoves~\cite[Theorem~3.1]{Vilkas2025}, and Tao~\cite{Tao2025}
studied paths. On paths, considerably shorter single-target strategies are
already known: Vilkas~\cite[Example~4.7]{Vilkas2025} showed that an optimal
strategy may be chosen using at most two complete hypermoves. Random pairwise
averaging has also been studied from the viewpoint of convergence and
mixing~\cite{AL2012,CDSZ2022,MSW2024}. These results do not provide a
strategy-length bound depending only on the number of vertices for arbitrary
finite graphs. Gollin et al.~\cite[Section~6]{GHH+2025} asked whether such a
computable bound exists, and in particular whether $2^{|V|}$ always suffices
\cite[Question~6.1]{GHH+2025}.

\subsection{Main results}

To obtain a bound uniform in the initial weights, we pass to the convex hull
of the strategy operators. For $d\ge0$, set
\[
 \Acal_d(G)=\{A_\sigma:|\sigma|\le d\},\qquad
 \PP_d(G)=\conv\Acal_d(G),\qquad
 \PP(G)=\conv\bigcup_{d\ge0}\Acal_d(G).
\]
All convex combinations are finite. Related convex hulls of states reachable
by pairwise averaging were studied by Hay, Schiff, and Fisch under the name
\emph{diffusion polytopes}~\cite{HSF2017}.

\begin{theorem}\label{thm:operators}
For every $n$ there is an explicit integer $L_n$ such that, for every
$n$-vertex graph $G$,
\[
 \PP(G)=\PP_{L_n}(G).
\]
For $n\ge2$, the sequence defined in \eqref{eq:bound} satisfies
\[
 L_n\le
 16^{n-2}\left(\frac{n!}{2}\right)^2
 \lceil\log_2 n\rceil^{n-2}
 =\exp\bigl(2n\log n+n\log\log n+O(n)\bigr),
\]
and hence $L_n\le n^{(2+o(1))n}$. In particular, $\PP(G)$ is a rational
polytope.
\end{theorem}

For the connected-set model considered here, the single-target optimization
problem is an immediate consequence.

\begin{corollary}\label{cor:water}
For every $n$-vertex graph $G$, every target $v_\ast\in V(G)$, and every
$\omega\in\R_{\ge0}^{V(G)}$,
\[
 t(G,v_\ast,\omega)\le L_n.
\]
\end{corollary}

Indeed, $A\mapsto(A\omega)_{v_\ast}$ is linear, so Theorem~\ref{thm:operators}
reduces the optimization to the finite set $\Acal_{L_n}(G)$.

Gollin et al.~\cite[Section~6]{GHH+2025} also suggested, as a route to
Question~6.1, proving that every instance admits an optimal strategy with no
repeated sharing move; Vilkas~\cite[Section~5]{Vilkas2025} likewise discusses
whether repetition can be forced in shortest optimal hypermove sequences. We
refer to the former assertion as the \emph{no-repetition conjecture} and
disprove it on five vertices. A strategy is \emph{repetition-free} if no
nontrivial connected set is averaged twice.

\begin{theorem}\label{thm:repetition}
Let
\[
 V=\{a,b,x,c,d\},\qquad E=\{ab,bc,bd,xc,xd\},
\]
let the target be $x$, and take initial weights
\[
 \omega=(2,5,1,0,0)
\]
in the order $(a,b,x,c,d)$. The unrestricted optimum is $65/32$, whereas the
maximum over repetition-free strategies is $2$. Thus every optimal strategy
repeats a connected averaging set. The minimum length of an optimal strategy
is $4$.
\end{theorem}

\subsection{Proof outline}

The proof of Theorem~\ref{thm:operators} is a shortening argument.
Disconnected windows reduce to smaller induced subgraphs. Connected windows
contract a quadratic energy, so sufficiently long blocks are close to the
complete averaging matrix $J$. Exact shortening is then reduced to a local
convex decomposition near $J$. Tree cuts provide rank-one coordinates; for
graphs other than paths, non-cut vertices give both signs of each coordinate,
whereas paths are handled by total nonnegativity.

Theorem~\ref{thm:repetition} is proved independently using two finite
rational invariant polytopes.

\section{Shortening long strategies}\label{sec:shortening}

We first reduce Theorem~\ref{thm:operators} to a local convex-decomposition
statement near complete averaging. Disconnected windows are shortened on
smaller graphs; connected windows drive the operator toward $J$.

\subsection{Local replacement}

We regard vectors as columns and write $e_v$ for the coordinate vector at
$v$. For a nonempty connected set $S\subseteq V$,
\[
 (A_Sz)_v=
 \begin{cases}
 |S|^{-1}\sum_{u\in S}z_u,&v\in S,\\
 z_v,&v\notin S.
 \end{cases}
\]
Thus $A_S$ is symmetric, idempotent, and doubly stochastic. Every positive
entry of $A_S$ is at least $1/n$.

Let
\[
 J=\frac1n\one\one^{\tr},
 \qquad
 \normmax{M}=\max_{i,j}|M_{ij}|.
\]
Every doubly stochastic matrix $M$ satisfies $MJ=JM=J$. Hence, for doubly
stochastic $P,Q$,
\begin{equation}\label{eq:nonexpansive}
 \normmax{PQ-J}\le \normmax{Q-J}.
\end{equation}
Indeed, each entry of $P(Q-J)$ is a convex combination of entries in the
corresponding column of $Q-J$.

Convex decompositions are stable under multiplication. If
\[
 M=\sum_{j=1}^r\lambda_jM_j,
 \qquad
 \lambda_j\ge0,
 \qquad
 \sum_{j=1}^r\lambda_j=1,
\]
and $W,U$ are the operators before and after this segment, then
\begin{equation}\label{eq:replacement}
 UMW=\sum_{j=1}^r\lambda_jUM_jW.
\end{equation}
It therefore suffices to express every sufficiently long strategy operator as
a finite convex combination of strictly shorter ones.

\subsection{Disconnected windows}

For a strategy $\sigma=(S_1,\ldots,S_k)$, its \emph{support graph} is the
graph on $V$ with edge set
\[
 \bigcup_{t=1}^k E(G[S_t]).
\]
A \emph{window} is a consecutive segment of a strategy. It is called
connected if its support graph is connected; vertices unused by the window
remain isolated.

\begin{lemma}\label{lem:disconnected-window}
Let $n\ge3$. Suppose
\[
 \ell_1=0,
 \qquad
 \ell_j\ge2\ell_{j-1}\quad(2\le j<n),
\]
and assume that, for every $j<n$, every strategy operator on every
$j$-vertex graph is a convex combination of strategy operators of length at
most $\ell_j$. Then every disconnected window on $n$ vertices is a convex
combination of strategy operators of length at most $\ell_{n-1}$. In
particular, every disconnected window of length $\ell_{n-1}+1$ can be
strictly shortened.
\end{lemma}

\begin{proof}
Let $C_1,\ldots,C_r$, with $r\ge2$, be the components of the support graph,
and set $s_i=|C_i|$. Every averaged set is contained in one component, while
operations supported on different components commute. Thus the window may be
reordered as a product of sub-strategies on the induced graphs $G[C_i]$.
Applying the hypothesis to each component gives a convex decomposition into
strategies of length at most
\[
 \sum_{i=1}^r\ell_{s_i}.
\]

We claim that, whenever $a,b\ge1$ and $a+b-1<n$,
\begin{equation}\label{eq:component-merge}
 \ell_a+\ell_b\le\ell_{a+b-1}.
\end{equation}
If $a=1$ or $b=1$, this is immediate from $\ell_1=0$. Otherwise, assuming
$a\le b$,
\[
 \ell_a+\ell_b\le2\ell_b\le\ell_{b+1}\le\ell_{a+b-1}.
\]
Repeated application of \eqref{eq:component-merge} gives
\[
 \sum_{i=1}^r\ell_{s_i}
 \le \ell_{s_1+\cdots+s_r-r+1}
 =\ell_{n-r+1}
 \le\ell_{n-1}.
\]
The resulting operator identity remains valid inside the original strategy by
\eqref{eq:replacement}.
\end{proof}

\subsection{Connected windows}

For $x\in\R^n$, set
\[
 \bar x=\frac1n\sum_{i=1}^n x_i,
 \qquad
 \mathcal E(x)=\sum_{i=1}^n(x_i-\bar x)^2.
\]
If $|S|=s$, then
\begin{equation}\label{eq:energy-drop}
 \mathcal E(x)-\mathcal E(A_Sx)
 =\frac1s\sum_{\substack{i<j\\i,j\in S}}(x_i-x_j)^2.
\end{equation}
Thus every averaging operation decreases $\mathcal E$. Connectedness gives a
uniform decrease over an entire window. The proof uses a first-crossing
argument similar to that in distributed averaging; see
\cite[Section~3.2, Lemma~4]{NOOT2009}.

\begin{lemma}\label{lem:window-contraction}
Let $R$ be the operator of a connected window on $n\ge2$ vertices. Then
\begin{equation}\label{eq:window-contraction}
 \mathcal E(Rx)
 \le\left(1-\frac4{n^2}\right)\mathcal E(x)
 \qquad(x\in\R^n).
\end{equation}
In particular, the contraction factor is independent of the length of the
window.
\end{lemma}

\begin{proof}
Index the vertices temporarily so that the initial values satisfy
\[
 a_1\le a_2\le\cdots\le a_n,
 \qquad
 g_p=a_{p+1}-a_p\quad(1\le p<n).
\]
The relabeling is used only in this proof. Let $x^{(t)}$ be the vector after
operation $t$, and put $U_p=\{1,\ldots,p\}$. Let $\tau_p$ be the first
operation whose averaged set meets both $U_p$ and its complement. Since the
window is connected, every $\tau_p$ exists.

Before time $\tau_p$, the two sides of the cut evolve independently. Hence
\begin{equation}\label{eq:uncrossed-cut}
 x_i^{(\tau_p-1)}\le a_p\quad(i\le p),
 \qquad
 x_j^{(\tau_p-1)}\ge a_{p+1}\quad(j>p).
\end{equation}
For an operation at time $t$, set $C_t=\{p:\tau_p=t\}$. If $i<j$ belong
to the averaged set $S_t$, let
\[
 F_{ij,t}=C_t\cap\{i,\ldots,j-1\}.
\]
If $F_{ij,t}\ne\varnothing$, then \eqref{eq:uncrossed-cut}, applied to the
smallest and largest indices in $F_{ij,t}$, gives
\[
 x_j^{(t-1)}-x_i^{(t-1)}
 \ge a_{\max F_{ij,t}+1}-a_{\min F_{ij,t}}
 \ge\sum_{p\in F_{ij,t}}g_p.
\]
Since $g_p\ge0$,
\[
 \bigl(x_j^{(t-1)}-x_i^{(t-1)}\bigr)^2
 \ge\sum_{p\in F_{ij,t}}g_p^2.
\]

Write $s=|S_t|$ and $m_{p,t}=|S_t\cap U_p|$. Substituting into
\eqref{eq:energy-drop} gives
\[
 \mathcal E(x^{(t-1)})-\mathcal E(x^{(t)})
 \ge
 \sum_{p\in C_t}
 \frac{m_{p,t}(s-m_{p,t})}{s}\,g_p^2.
\]
For $p\in C_t$, we have $1\le m_{p,t}<s$, so
\[
 \frac{m_{p,t}(s-m_{p,t})}{s}\ge\frac12.
\]
Each cut is charged exactly once. Summing over the window therefore yields
\begin{equation}\label{eq:gap-charge}
 \mathcal E(x)-\mathcal E(Rx)
 \ge\frac12\sum_{p=1}^{n-1}g_p^2.
\end{equation}

On the other hand,
\begin{align*}
 \mathcal E(x)
 &=\frac1n\sum_{i<j}(a_j-a_i)^2\\
 &\le\frac1n\sum_{i<j}(j-i)\sum_{p=i}^{j-1}g_p^2\\
 &=\frac12\sum_{p=1}^{n-1}p(n-p)g_p^2\\
 &\le\frac{n^2}{8}\sum_{p=1}^{n-1}g_p^2.
\end{align*}
Combining this with \eqref{eq:gap-charge} proves
\eqref{eq:window-contraction}.
\end{proof}

Iterating this estimate gives the required mixing bound.

\begin{lemma}\label{lem:mixing}
Let $n\ge3$, and let $R$ be the operator of a concatenation of $q\ge1$
connected windows of arbitrary lengths. Then
\begin{equation}\label{eq:mixing-rate}
 \normmax{R-J}
 \le\|R-J\|_{2\to2}
 \le\left(1-\frac4{n^2}\right)^{q/2}
 \le e^{-2q/n^2}.
\end{equation}
\end{lemma}

\begin{proof}
Every averaging operator fixes $\one$ and preserves $\one^\perp$.
Lemma~\ref{lem:window-contraction} therefore bounds the Euclidean operator
norm of each connected-window operator on $\one^\perp$ by
$(1-4/n^2)^{1/2}$. Multiplying these bounds gives the middle inequality.
The first follows from $|M_{ij}|\le\|M\|_{2\to2}$, and the last from
$1-u\le e^{-u}$.
\end{proof}

\subsection{The local problem}

Closeness to $J$ alone does not imply finite convex stabilization. Indeed, for $0<\lambda<1$,
\[
 M=\lambda I+(1-\lambda)J
\]
satisfies
\[
 M^k=J+\lambda^k(I-J)\longrightarrow J,
\]
but
\[
 M^{d+1}\notin\conv\{I,M,\ldots,M^d\}
\]
for every $d$. Thus we need an exact convex decomposition near $J$.

For the global induction it is enough to decompose the product of two
consecutive mixed blocks.

\begin{proposition}\label{prop:local-shortening}
Let $G$ be a connected graph on $n\ge3$ vertices and put
\[
 \delta_n=\frac1{4n^5}.
\]
If $P,Q$ are strategy operators on $G$ such that
\[
 \normmax{P-J},\ \normmax{Q-J}\le\delta_n,
\]
then
\[
 QP\in\PP_6(G).
\]
If $G$ is a path, then in fact $QP\in\PP_4(G)$.
\end{proposition}

We postpone the proof to Sections~\ref{sec:coordinates}--\ref{sec:path} and
first derive the global bound.

\subsection{The uniform bound}

Define
\begin{equation}\label{eq:bound}
\begin{split}
 L_1&=0,\qquad L_2=1,\\
 r_n&=\lceil\log_2 n\rceil,\qquad q_n=4n^2r_n,\\
 L_n&=2q_n(L_{n-1}+1)\qquad(n\ge3).
\end{split}
\end{equation}
Since $q_n\ge1$, we have $L_n\ge2L_{n-1}$.

\begin{proof}[Proof of Theorem~\ref{thm:operators}, assuming
Proposition~\ref{prop:local-shortening}]
We argue by induction on $n$. The cases $n=1$ and $n=2$ are immediate; for
$n=2$, the only nontrivial operation averages both vertices and is
idempotent.

Fix $n\ge3$ and assume the theorem for all smaller graphs. Put
\[
 h=L_{n-1}+1,
 \qquad
 q=q_n,
 \qquad
 N=2qh=L_n.
\]
Let $\sigma$ be a strategy of length $k\ge N$, and partition its first $N$
operations into $2q$ consecutive windows of length $h$.

If one window is disconnected, Lemma~\ref{lem:disconnected-window}, with
$\ell_j=L_j$, replaces it by a convex combination of strategy operators of
length at most $h-1$. By \eqref{eq:replacement}, $A_\sigma$ is then a convex
combination of strategy operators of length at most $k-1$.

Suppose all $2q$ windows are connected. Split the prefix into two
chronological halves, with operators $P$ and $Q$, so that the prefix operator
is $QP$. Each half contains $q$ connected windows. Since
$r_n\ge\log_2 n\ge\log n$, Lemma~\ref{lem:mixing} gives
\[
 \normmax{P-J},\ \normmax{Q-J}
 \le e^{-2q/n^2}
 =e^{-8r_n}
 \le n^{-8}
 \le\frac1{4n^5}
 =\delta_n.
\]
Proposition~\ref{prop:local-shortening} replaces $QP$ by a convex combination
of strategy operators of length at most six. Retaining the remaining $k-N$
operations gives strategies of length at most
\[
 k-N+6<k.
\]
Thus every strategy of length at least $N$ is a finite convex combination of
strictly shorter strategies. Strong induction on the length gives
\[
 \PP(G)=\PP_{L_n}(G).
\]
\end{proof}

For $n\ge3$, $L_{n-1}\ge1$, and therefore
\[
 L_n
 =2q_n(L_{n-1}+1)
 \le16n^2r_nL_{n-1}.
\]
Iterating and using $r_j\le r_n$ for $j\le n$ yields
\begin{equation}\label{eq:factorial-bound}
 L_n
 \le16^{n-2}\left(\frac{n!}{2}\right)^2
          \lceil\log_2n\rceil^{n-2}
 =\exp\bigl(2n\log n+n\log\log n+O(n)\bigr).
\end{equation}
Hence $L_n\le n^{(2+o(1))n}$. The remainder of the proof is devoted to
Proposition~\ref{prop:local-shortening}.

\section{Cut coordinates and rank-one operators}\label{sec:coordinates}

We next construct short strategy operators whose deviations from $J$ span the
relevant affine space. Connected cuts give rank-one directions, and the cuts
of a spanning tree give a basis.

\subsection{Two-block projections}

Let \(\varnothing\ne S\subsetneq V\), and assume that both \(G[S]\) and
\(G[V\setminus S]\) are connected. Write \(\one_S\) for the indicator of
\(S\), and set
\[
 u_S=\one_S-\frac{|S|}{n}\one,
 \qquad
 s_S=\langle u_S,u_S\rangle=\frac{|S|(n-|S|)}n.
\]
Average separately on the two sides of the cut:
\[
 E_S=A_SA_{V\setminus S}.
\]
The two factors commute, and \(E_S\) is the orthogonal projection onto the
space of vectors that are constant on each side of the cut. Since this space
is spanned by the orthogonal vectors \(\one\) and \(u_S\),
\begin{equation}\label{eq:cut-projection}
 E_S=J+\frac{u_Su_S^{\tr}}{s_S}.
\end{equation}
Thus a two-block average differs from \(J\) in a single rank-one direction.
For two such cuts \(S,T\),
\begin{equation}\label{eq:cut-product}
 E_SE_T-J
 =\frac{\langle u_S,u_T\rangle}{s_Ss_T}\,u_Su_T^{\tr},
\end{equation}
where
\begin{equation}\label{eq:cut-inner-product}
 \langle u_S,u_T\rangle
 =|S\cap T|-\frac{|S||T|}{n}.
\end{equation}
The operator \(E_SE_T\) is represented by at most four averaging operations.

\subsection{Tree coordinates}

Fix a rooted spanning tree \(T\) of the connected graph \(G\), with every
edge oriented away from the root. For \(e\in E(T)\), let \(v_e\) be its child,
\(p_e\) its parent, and \(S_e\) the vertex set of the component of \(T-e\)
containing \(v_e\). Both sides of this cut induce connected subgraphs of
\(G\). Set
\[
 u_e=u_{S_e},\qquad s_e=s_{S_e},\qquad E_e=E_{S_e}.
\]

Let \(H\) be the \(n\times(n-1)\) matrix with columns \(u_e\), and define the
tree-difference matrix \(D\) by
\[
 (Dz)_e=z_{v_e}-z_{p_e}.
\]
We also write
\[
 \mathcal A_V
 =\{M\in\R^{V\times V}:M\one=\one,\ \one^{\tr}M=\one^{\tr}\},
\]
\[
 \mathcal L_V
 =\{X\in\R^{V\times V}:X\one=0,\ \one^{\tr}X=0\}.
\]
Thus \(\mathcal A_V=J+\mathcal L_V\).

\begin{lemma}\label{lem:tree-coordinates}
The matrices \(D,H\) satisfy
\[
 DH=I_{n-1},\qquad HD=I_n-J.
\]
Consequently, for every \(M\in\mathcal A_V\),
\begin{equation}\label{eq:reconstruct}
 M-J
 =H(DMD^{\tr})H^{\tr}
 =\sum_{e,f\in E(T)}(DMD^{\tr})_{ef}\,u_eu_f^{\tr}.
\end{equation}
Hence the matrices
\[
 \{u_eu_f^{\tr}:e,f\in E(T)\}
\]
form a basis of \(\mathcal L_V\).
\end{lemma}

\begin{proof}
For tree edges \(e,f\), the indicator \(\one_{S_f}\) changes from parent to
child across \(e\) exactly when \(e=f\). Since the centering term in \(u_f\)
is constant,
\[
 (Du_f)_e=\delta_{ef}.
\]
Thus \(DH=I_{n-1}\). The columns of \(H\) are therefore linearly independent;
since they are centered, they form a basis of \(\one^\perp\).

For \(z\in\R^V\),
\[
 D(z-HDz)=0.
\]
Hence \(z-HDz\) is constant on the tree. Since \(HDz\in\one^\perp\), this
constant is the mean of the coordinates of \(z\), and therefore
\[
 HDz=z-\frac{\one^{\tr}z}{n}\one=(I-J)z.
\]
Thus \(HD=I-J\).

If \(M\in\mathcal A_V\), then \(MJ=JM=J\). Hence
\[
 H(DMD^{\tr})H^{\tr}
 =(I-J)M(I-J)
 =M-J,
\]
which proves \eqref{eq:reconstruct}. Finally, the map
\[
 B\longmapsto HBH^{\tr}
\]
is injective because
\[
 D(HBH^{\tr})D^{\tr}=B,
\]
and is surjective onto \(\mathcal L_V\) by \eqref{eq:reconstruct}. The stated
rank-one matrices therefore form a basis of \(\mathcal L_V\).
\end{proof}

The matrix \(DMD^{\tr}\) gives the cut coordinates of \(M-J\). It remains
to realize both signs of these coordinates by short strategy operators; the
argument depends on whether \(G\) is a path.

\section{Graphs other than paths: sign reversal}\label{sec:nonpath}

Section~\ref{sec:coordinates} realizes one nonzero multiple of each basis
direction \(u_eu_f^{\tr}\). If \(G\) is not a path, three non-cut vertices
provide the opposite sign.

\subsection{A sign-reversal identity}

A vertex \(a\in V(G)\) is \emph{non-cut} if
\(G[V\setminus\{a\}]\) is connected. For such a vertex, the singleton cut
\(\{a\},V\setminus\{a\}\) gives
\[
 u_a=e_a-\frac1n\one,
 \qquad
 s_a=\frac{n-1}{n},
 \qquad
 E_a=A_{V\setminus\{a\}}.
\]
Thus \(E_a\) is a one-operation strategy operator. For distinct vertices
\(a,b\),
\begin{equation}\label{eq:singleton-pairing}
 \langle u_a,u_b\rangle=-\frac1n.
\end{equation}

\begin{lemma}\label{lem:three-vertices}
Every finite connected graph that is not a path has at least three distinct
non-cut vertices.
\end{lemma}

\begin{proof}
If every vertex has degree at most two, then the graph is a path or a cycle;
in the latter case every vertex is non-cut. Otherwise, choose three edges
incident with a vertex of degree at least three and extend them to a spanning
tree \(T\). The tree \(T\) has at least three leaves. If \(a\) is a leaf,
then \(T-a\) is a connected spanning subgraph of \(G-a\), so \(a\) is
non-cut.
\end{proof}

Fix a tree edge \(f\). Among three non-cut vertices, two lie on the same side
of the cut \(S_f\). For such vertices \(a,b\),
\begin{equation}\label{eq:same-side-pairing}
 \langle u_a,u_f\rangle=\langle u_b,u_f\rangle.
\end{equation}
Indeed, both sides equal \(1-|S_f|/n\) if \(a,b\in S_f\), and
\(-|S_f|/n\) otherwise.

\begin{lemma}\label{lem:same-side-sign-reversal}
Let \(e,f\) be tree edges, and let \(a,b\) be distinct non-cut vertices on
the same side of \(S_f\). Define
\[
 U_{ef}=E_eE_aE_f,
 \qquad
 V_{ef}=E_eE_aE_bE_f.
\]
Then
\begin{equation}\label{eq:two-signs}
 U_{ef}-J=c_{ef}u_eu_f^{\tr},
 \qquad
 V_{ef}-J=-\frac{c_{ef}}{n-1}u_eu_f^{\tr},
\end{equation}
where
\begin{equation}\label{eq:cef}
 c_{ef}
 =\frac{\langle u_e,u_a\rangle\langle u_a,u_f\rangle}
        {s_es_as_f}.
\end{equation}
Moreover,
\begin{equation}\label{eq:cef-lower}
 |c_{ef}|\ge\frac{n}{(n-1)^3},
 \qquad
 \frac{|c_{ef}|}{n-1}\ge\frac{n}{(n-1)^4}\ge n^{-3}.
\end{equation}
The operators \(U_{ef}\) and \(V_{ef}\) use at most five and six averaging
operations, respectively.
\end{lemma}

\begin{proof}
On \(\one^\perp\), the two-block projection \(E_x\) acts as
\(u_xu_x^{\tr}/s_x\). Hence
\[
 U_{ef}-J
 =\frac{\langle u_e,u_a\rangle\langle u_a,u_f\rangle}
        {s_es_as_f}\,u_eu_f^{\tr}.
\]
Inserting \(E_b\) multiplies the coefficient by
\[
 \frac{\langle u_a,u_b\rangle}{s_b}
 \frac{\langle u_b,u_f\rangle}{\langle u_a,u_f\rangle}
 =-\frac1{n-1},
\]
by \eqref{eq:singleton-pairing} and \eqref{eq:same-side-pairing}. This proves
\eqref{eq:two-signs}.

For any tree edge \(e\),
\begin{equation}\label{eq:tree-singleton-ratio}
 \frac{\langle u_e,u_a\rangle}{s_e}
 =\begin{cases}
   |S_e|^{-1},&a\in S_e,\\
   -(n-|S_e|)^{-1},&a\notin S_e.
  \end{cases}
\end{equation}
The same formula holds with \(e\) replaced by \(f\). Since
\(s_a=(n-1)/n\), equations \eqref{eq:cef} and
\eqref{eq:tree-singleton-ratio} give
\[
 |c_{ef}|\ge\frac{n}{(n-1)^3}.
\]
The remaining inequalities in \eqref{eq:cef-lower} follow immediately.
Finally, \(E_e\) and \(E_f\) each use at most two operations, while
\(E_a\) and \(E_b\) use one each.
\end{proof}

\subsection{A neighborhood of complete averaging}

We apply the two signs coordinatewise.

\begin{proposition}\label{prop:nonpath-neighborhood}
Let \(G\) be a connected graph on \(n\) vertices that is not a path. If
\(M\in\mathcal A_V\) and
\[
 \normmax{M-J}\le\frac1{4n^5},
\]
then \(M\in\PP_6(G)\). Hence \(J\) is an interior point of
\(\PP_6(G)\) relative to \(\mathcal A_V\).
\end{proposition}

\begin{proof}
Choose three distinct non-cut vertices by Lemma~\ref{lem:three-vertices}.
For each tree edge \(f\), choose two of them, denoted \(a_f,b_f\), on the
same side of \(S_f\). Lemma~\ref{lem:same-side-sign-reversal} then provides,
for every ordered pair \(e,f\), two strategy operators in the direction
\(u_eu_f^{\tr}\) with opposite signs and coefficient magnitudes at least
\(n^{-3}\).

Set
\[
 B=D(M-J)D^{\tr}.
\]
Each entry \(B_{ef}\) is a signed sum of four entries of \(M-J\), so
\begin{equation}\label{eq:B-bound-nonpath}
 |B_{ef}|\le4\normmax{M-J}.
\end{equation}
For each \(e,f\), choose \(R_{ef}\in\{U_{ef},V_{ef}\}\) so that
\[
 R_{ef}-J=d_{ef}u_eu_f^{\tr},
 \qquad
 B_{ef}d_{ef}\ge0,
 \qquad
 |d_{ef}|\ge n^{-3},
\]
and put \(\lambda_{ef}=B_{ef}/d_{ef}\). If \(B_{ef}=0\), set
\(\lambda_{ef}=0\). Then
\[
 \sum_{e,f}\lambda_{ef}
 \le n^3\sum_{e,f}|B_{ef}|
 \le4n^3(n-1)^2\normmax{M-J}
 \le1.
\]
Using \eqref{eq:reconstruct},
\[
 M-J
 =\sum_{e,f}B_{ef}u_eu_f^{\tr}
 =\sum_{e,f}\lambda_{ef}(R_{ef}-J).
\]
Therefore
\[
 M
 =\left(1-\sum_{e,f}\lambda_{ef}\right)J
  +\sum_{e,f}\lambda_{ef}R_{ef}.
\]
Since \(J=A_V\) uses one operation and every \(R_{ef}\) uses at most six,
this is a convex combination in \(\PP_6(G)\).

Finally, \(1/(4n^5)<1/n\), so the max-norm ball of radius \(1/(4n^5)\)
about \(J\) in \(\mathcal A_V\) consists of entrywise positive, hence doubly
stochastic, matrices. This proves the relative interior statement.
\end{proof}

For paths, this sign-reversal argument is unavailable; we use total
nonnegativity instead.

\section{Paths and total nonnegativity}\label{sec:path}

A path has only two non-cut vertices, so the preceding sign-reversal argument
does not apply. Instead, we use a one-sided cone of rank-one directions and
show that the product of two sufficiently mixed path strategy operators lies
in this cone.

\subsection{Prefix cuts and mixed differences}

Label the vertices of $P_n$ by $1,\ldots,n$ in path order, where $n\ge2$.
For $1\le p<n$, set
\[
 u_p=\one_{[1,p]}-\frac pn\one,
 \qquad
 s_p=\frac{p(n-p)}n,
 \qquad
 E_p=A_{[1,p]}A_{[p+1,n]}.
\]
Let $H$ have columns $u_p$, and let $D$ be the path-difference matrix
\[
 (Dz)_p=z_p-z_{p+1}.
\]
These are the tree coordinates of Section~\ref{sec:coordinates}, rooted at
$n$. For a doubly stochastic matrix $M$, write
\[
 B(M)=DMD^{\tr}.
\]
Thus
\[
 B(M)_{pq}
 =M_{pq}-M_{p+1,q}-M_{p,q+1}+M_{p+1,q+1}.
\]
We call $B(M)$ the \emph{mixed-difference matrix}.

The prefix cuts are nested. If $p\le q$, then
\begin{equation}\label{eq:prefix-inner}
 \langle u_p,u_q\rangle=\frac{p(n-q)}n>0,
 \qquad
 \frac{s_ps_q}{\langle u_p,u_q\rangle}
 =\frac{q(n-p)}n\le n,
\end{equation}
and the same bound holds after interchanging $p$ and $q$. Hence
$E_pE_q-J$ is a positive multiple of $u_pu_q^{\tr}$. This makes entrywise
nonnegativity of $B(M)$ exactly the sign condition needed for a convex
decomposition into short path strategies.

\begin{lemma}\label{lem:path-convex}
Let $M$ be doubly stochastic. If $B(M)$ is entrywise nonnegative and
\[
 \normmax{M-J}\le\frac1{4n},
\]
then $M\in\PP_4(P_n)$. More precisely,
\begin{equation}\label{eq:path-convex}
 M=(1-\Lambda)J+
 \sum_{p,q=1}^{n-1}\lambda_{pq}E_pE_q,
 \qquad
 \lambda_{pq}
 =B(M)_{pq}\frac{s_ps_q}{\langle u_p,u_q\rangle},
\end{equation}
where
\[
 \Lambda:=\sum_{p,q=1}^{n-1}\lambda_{pq}
 \le4n\normmax{M-J}\le1.
\]
\end{lemma}

\begin{proof}
By \eqref{eq:reconstruct},
\[
 M-J=\sum_{p,q=1}^{n-1}B(M)_{pq}\,u_pu_q^{\tr}.
\]
Combining this with \eqref{eq:cut-product} gives
\eqref{eq:path-convex}. The coefficients $\lambda_{pq}$ are nonnegative by
hypothesis and \eqref{eq:prefix-inner}. Moreover,
\[
 \Lambda
 \le n\sum_{p,q=1}^{n-1}B(M)_{pq}.
\]
The double sum telescopes:
\[
 \sum_{p,q=1}^{n-1}B(M)_{pq}
 =M_{11}-M_{1n}-M_{n1}+M_{nn}.
\]
Since the left-hand side is nonnegative,
\[
 0\le\sum_{p,q}B(M)_{pq}
 \le4\normmax{M-J}.
\]
Thus $\Lambda\le1$, and \eqref{eq:path-convex} is a convex combination.
Finally, $J=A_{V(P_n)}$ uses one averaging operation and each $E_pE_q$ uses
at most four.
\end{proof}

\subsection{Total nonnegativity}

A matrix is \emph{totally nonnegative} if all its minors are nonnegative.

\begin{lemma}\label{lem:TN}
Every strategy operator on $P_n$ is totally nonnegative.
\end{lemma}

\begin{proof}
Every connected subset of a path is an interval. Its averaging matrix is
block diagonal in path order, with identity blocks and one constant positive
rank-one block, and is therefore totally nonnegative. Products of totally
nonnegative matrices remain totally nonnegative by the Cauchy--Binet formula.
\end{proof}

Let $P,Q$ be positive doubly stochastic matrices. For $1\le i,j<n$, define
\[
 d_k=P_{ik}-P_{i+1,k},
 \qquad
 e_k=Q_{kj}-Q_{k,j+1}.
\]
Then
\begin{equation}\label{eq:mixed-product}
 B(PQ)_{ij}
 =\sum_{k=1}^n d_ke_k
 =\langle d,e\rangle.
\end{equation}
If $P$ and $Q$ are totally nonnegative, their $2\times2$ minors imply, for
$k<\ell$,
\[
 \frac{P_{ik}}{P_{i+1,k}}
 \ge
 \frac{P_{i\ell}}{P_{i+1,\ell}},
 \qquad
 \frac{Q_{kj}}{Q_{k,j+1}}
 \ge
 \frac{Q_{\ell j}}{Q_{\ell,j+1}}.
\]
Thus the relevant row and column ratios are monotone. Near $J$, the
difference vectors $d$ and $e$ are therefore close to the centered monotone
cone.

\subsection{The centered monotone cone}

Set
\[
 K_n=\{z\in\R^n:z_1\ge\cdots\ge z_n,\ \one^{\tr}z=0\}.
\]

\begin{lemma}\label{lem:relative}
Let $a,b$ be strictly positive probability vectors such that $a_j/b_j$ is
nonincreasing in $j$. Suppose
\[
 |nb_j-1|\le\eta<1
 \qquad(1\le j\le n).
\]
If $d=a-b$, then there exists $z\in K_n$ such that
\[
 \normone{d-z}
 \le\frac{2\eta}{1-\eta}\normone d.
\]
\end{lemma}

\begin{proof}
Set
\[
 t_j=\frac{a_j}{b_j}-1,
 \qquad
 \bar t=\frac1n\sum_{j=1}^n t_j,
 \qquad
 z=\frac1n(t-\bar t\one).
\]
Since $(t_j)$ is nonincreasing, $z\in K_n$. Also,
$\sum_jb_jt_j=\sum_jd_j=0$, so
\[
 |\bar t|
 =\left|\sum_j(1/n-b_j)t_j\right|
 \le\frac\eta n\sum_j|t_j|.
\]
Since $b_j\ge(1-\eta)/n$,
\[
 \sum_j|t_j|
 \le\frac n{1-\eta}\normone d.
\]
Finally,
\[
 d_j-z_j=(b_j-1/n)t_j+\frac{\bar t}{n},
\]
and summing absolute values gives the stated estimate.
\end{proof}

\begin{lemma}\label{lem:angle}
For all $z,w\in K_n$,
\[
 \langle z,w\rangle
 \ge\frac1{4n}\normone z\normone w.
\]
\end{lemma}

\begin{proof}
Every $z\in K_n$ has the expansion
\[
 z=\sum_{p=1}^{n-1}(z_p-z_{p+1})u_p
\]
with nonnegative coefficients. Since
\[
 \normone{u_p}=\frac{2p(n-p)}n,
\]
for $p\le q$ equation \eqref{eq:prefix-inner} gives
\[
 \frac{\langle u_p,u_q\rangle}
      {\normone{u_p}\normone{u_q}}
 =\frac n{4q(n-p)}
 \ge\frac1{4n},
\]
and the same bound holds for $q\le p$. Expanding both vectors and applying
this estimate term by term yields
\[
 \langle z,w\rangle
 \ge\frac1{4n}
 \left(\sum_pa_p\normone{u_p}\right)
 \left(\sum_qb_q\normone{u_q}\right),
\]
where $a_p,b_q\ge0$ are the expansion coefficients. The triangle inequality
then gives the result.
\end{proof}

\subsection{The two-factor estimate}

\begin{proposition}\label{prop:two-factor}
Let $P,Q$ be totally nonnegative doubly stochastic $n\times n$ matrices,
where $n\ge2$. If
\[
 \normmax{P-J},\ \normmax{Q-J}
 \le\frac1{64n^2},
\]
then $B(PQ)$ is entrywise nonnegative and
\[
 PQ\in\PP_4(P_n).
\]
\end{proposition}

\begin{proof}
The hypothesis makes every entry of $P$ and $Q$ positive. Fix
$1\le i,j<n$, and let $d,e$ be the vectors in \eqref{eq:mixed-product}.
Their coordinate sums are zero, and total nonnegativity gives the monotone
ratio hypotheses of Lemma~\ref{lem:relative}.

Since
\[
 |nP_{i+1,k}-1|,
 \ |nQ_{k,j+1}-1|
 \le\frac1{64n},
\]
Lemma~\ref{lem:relative}, with $\eta=1/(64n)$, gives $z,w\in K_n$ such that
\[
 \normone{d-z}\le\kappa\normone d,
 \qquad
 \normone{e-w}\le\kappa\normone e,
 \qquad
 \kappa=\frac2{64n-1}\le\frac1{16n}.
\]
Hence
\[
 \normone z\ge(1-\kappa)\normone d,
 \qquad
 \normone w\ge(1-\kappa)\normone e.
\]
Lemma~\ref{lem:angle} therefore yields
\[
 \langle z,w\rangle
 \ge\frac{(1-\kappa)^2}{4n}\normone d\normone e.
\]
Moreover,
\[
 \langle d,e\rangle-\langle z,w\rangle
 =\langle d-z,e\rangle+\langle z,e-w\rangle.
\]
Using $|\langle x,y\rangle|\le\normone x\normone y$ and
$\normone z\le(1+\kappa)\normone d$, we obtain
\begin{align*}
 \langle d,e\rangle
 &\ge
 \left(
 \frac{(1-\kappa)^2}{4n}-2\kappa-\kappa^2
 \right)\normone d\normone e\\
 &\ge
 \left(\frac1{8n}-\frac9{256n^2}\right)
 \normone d\normone e
 \ge\frac1{16n}\normone d\normone e
 \ge0.
\end{align*}
Thus $B(PQ)$ is entrywise nonnegative by \eqref{eq:mixed-product}. Finally,
\eqref{eq:nonexpansive} gives
\[
 \normmax{PQ-J}
 \le\frac1{64n^2}
 \le\frac1{4n},
\]
so Lemma~\ref{lem:path-convex} gives $PQ\in\PP_4(P_n)$.
\end{proof}

\begin{proof}[Proof of Proposition~\ref{prop:local-shortening}]
If $G$ is not a path, then by \eqref{eq:nonexpansive},
\[
 \normmax{QP-J}
 \le\normmax{P-J}
 \le\delta_n,
\]
and Proposition~\ref{prop:nonpath-neighborhood} gives
$QP\in\PP_6(G)$.

Suppose that $G$ is a path and label its vertices in path order. By
Lemma~\ref{lem:TN}, both $P$ and $Q$ are totally nonnegative. Since
\[
 \delta_n=\frac1{4n^5}\le\frac1{64n^2}
 \qquad(n\ge3),
\]
Proposition~\ref{prop:two-factor}, applied to $(Q,P)$, gives
$QP\in\PP_4(G)$.
\end{proof}

This completes the proof of Theorem~\ref{thm:operators}.

\section{A five-vertex repetition counterexample}\label{sec:repetition}

Consider the four-cycle
\[
 b-c-x-d-b
\]
with a pendant vertex $a$ adjacent to $b$. The target is $x$, and the initial
weights, in the order $(a,b,x,c,d)$, are
\[
 \omega=(2,5,1,0,0).
\]

\subsection{An optimal strategy}

Consider the four operations
\[
 \{x,c\},\qquad \{b,c\},\qquad \{a,b,x,d\},\qquad \{x,c\}.
\]
The third set induces the path $a-b-d-x$. The successive states are
\begin{center}
\small
\setlength{\tabcolsep}{5pt}
\renewcommand{\arraystretch}{1.08}
\begin{tabular}{c@{\qquad}ccccc}
\toprule
 & $a$ & $b$ & $x$ & $c$ & $d$\\
\midrule
initial
  & $2$ & $5$ & $1$ & $0$ & $0$\\
after $\{x,c\}$
  & $2$ & $5$ & $1/2$ & $1/2$ & $0$\\
after $\{b,c\}$
  & $2$ & $11/4$ & $1/2$ & $11/4$ & $0$\\
after $\{a,b,x,d\}$
  & $21/16$ & $21/16$ & $21/16$ & $11/4$ & $21/16$\\
after $\{x,c\}$
  & $21/16$ & $21/16$ & $65/32$ & $65/32$ & $21/16$\\
\bottomrule
\end{tabular}
\end{center}
Thus the unrestricted optimum is at least $65/32$. A repetition-free
strategy attains $2$ by averaging once on $\{b,x,c\}$.

\subsection{A finite certificate}

For a chronological strategy $\sigma=(S_1,\ldots,S_k)$, symmetry of the
averaging matrices gives
\[
 (A_{S_k}\cdots A_{S_1}\omega)_x
 =\langle \omega,A_{S_1}\cdots A_{S_k}e_x\rangle.
\]
Hence we may work with reverse profiles from $e_x$ and the linear functional
\[
 F(p)=2p_a+5p_b+p_x.
\]
Let $A=A_{\{x,c\}}$. The symmetry exchanging $c$ and $d$ fixes the graph,
$e_x$, and $F$.

We use two finite rational invariant polytopes. The complete supplementary
certificate is part of the proof: it lists all rational generating points and
all convex identities used below, together with an independent verifier using
exact rational arithmetic.

\begin{lemma}\label{lem:repetition-certificate}
There exist rational polytopes \(K,H\subset\mathbb R^5\) such that
\begin{equation}\label{eq:rep-fullinv}
e_x\in K\cap H,
\qquad
A_SK\subseteq K
\quad\text{for every connected \(S\)},
\end{equation}
and
\begin{equation}\label{eq:rep-restinv}
A_SH\subseteq H
\quad\text{for every connected \(S\ne\{x,c\}\)}.
\end{equation}
For these polytopes,
\begin{align}
\max_{p\in K} F(p) &= \frac{65}{32},
\label{eq:rep-fullbound}\\
\max_{q\in H} F(Aq) &= 2.
\label{eq:rep-restbound}
\end{align}
In addition, for every connected \(S\) with \(|S|\ge 2\) and $S\notin\bigl\{\{x,c\},\{x,d\}\bigr\}$,
we have
\begin{equation}\label{eq:rep-lastbound}
\max_{p\in K} F(A_Sp)\le 2.
\end{equation}
Finally,
\begin{equation}\label{eq:rep-depth3}
\max_{\substack{
0\le k\le3\\
S_1,\ldots,S_k\ \text{nontrivial connected}
}}
F(A_{S_1}\cdots A_{S_k}e_x)
=2.
\end{equation}
\end{lemma}
\begin{proof}[Certificate verification]
The polytope $K$ has $27$ listed generating points and $H$ has $30$. For
each relevant connected set $S$, the supplement gives exact rational convex
identities certifying the required invariance: $432$ identities for $K$ and
$450$ for $H$. The optimization statements are checked on the listed
generating points. The supplied program \texttt{verify\_certificate.py} also
enumerates every reverse profile reachable in at most three nontrivial steps,
which proves \eqref{eq:rep-depth3}. It uses only exact rational arithmetic; no
floating-point tolerance or optimization solver enters the verification.
\end{proof}

\begin{proof}[Proof of Theorem~\ref{thm:repetition}]
By \eqref{eq:rep-fullinv}, every reverse profile belongs to $K$. Hence
\eqref{eq:rep-fullbound} gives the unrestricted upper bound $65/32$, attained
by the four-step strategy above.

Let the reverse word be repetition-free, and remove singleton moves. If no
nontrivial move remains, then the value is $F(e_x)=1\le2$. Otherwise, let $S$
be the last nontrivial averaged set. If $S$ is neither $\{x,c\}$ nor
$\{x,d\}$, the preceding profile lies in $K$, and
\eqref{eq:rep-lastbound} gives final value at most $2$.

If $S=\{x,c\}$, then no earlier move used this set. Starting from $e_x$,
\eqref{eq:rep-restinv} therefore keeps every preceding profile in $H$, and
\eqref{eq:rep-restbound} again gives final value at most $2$. The case
$S=\{x,d\}$ follows by exchanging $c$ and $d$.

Thus every repetition-free strategy has value at most $2$, while the one-step
strategy on $\{b,x,c\}$ attains $2$. Finally, \eqref{eq:rep-depth3} shows
that no strategy of length at most three can attain $65/32$, whereas the
four-step strategy above does. Hence the minimum length of an optimal strategy
is $4$. This proves the theorem.
\end{proof}

\medskip
\noindent\textbf{Supplementary material.}
The supplementary archive contains the complete rational certificate, the
exact verification program, its output, and a self-contained note recording
the certificate. These files verify all $882$ convex identities and the
depth-three enumeration used above.

\section*{Acknowledgments}
The authors developed the proof strategy and the principal mathematical ideas in this paper. GPT6 Astra played a key role in Section 4, where Proposition \ref{prop:nonpath-neighborhood} is established. Its most important contribution was the discovery that a matrix sufficiently close to $J$ can be written as a convex combination of short strategy operators. GPT6 Astra was also used for drafting and language editing, exploratory calculations, and assistance with computational checks. The authors independently verified every mathematical argument, numerical value, exact rational identity, and computer-assisted certificate appearing in the paper, and take full responsibility for the contents.

We thank Hehui Wu and Ningyuan Yang for their early involvement in discussions and for their suggestions.

\bibliographystyle{amsplain}
\bibliography{ref}

@article{MSW2024,
  author = {Movassagh, Ramis and Szegedy, Mario and Wang, Guanyang},
  title = {Repeated Averages on Graphs},
  journal = {The Annals of Applied Probability},
  volume = {34}, number = {4}, pages = {3781--3819},
  year = {2024}, doi = {10.1214/24-AAP2050}
}

@unpublished{GHH+2025,
  author = {Gollin, J. Pascal and Hendrey, Kevin and Huang, Hao and Huynh, Tony and Mohar, Bojan and Oum, Sang-il and Yang, Ningyuan and Yu, Wei-Hsuan and Zhu, Xuding},
  title = {Sharing Tea on a Graph},
  note = {Accepted for publication in {Combinatorial Theory}; arXiv:2405.15353v2},
  year = {2026}
}

@unpublished{HH2015,
  author = {H{\"a}ggstr{\"o}m, Olle and Hirscher, Timo},
  title = {Water Transport on Graphs},
  note = {arXiv:1504.03978}, year = {2015}
}

@article{Tao2025,
  author = {Tao, Tianyi},
  title = {Water Transport on a Path: Finding the Strategy Through Its Existence},
  journal = {Advances and Applications in Discrete Mathematics},
  volume = {42}, number = {5}, pages = {471--480},
  year = {2025}, doi = {10.17654/0974165825031}
}

@unpublished{Vilkas2025,
  author = {Vilkas, Timo}, title = {Water Transport on Finite Graphs},
  note = {arXiv:2501.16911}, year = {2025}
}

@article{NOOT2009,
  author = {Nedi\'c, Angelia and Olshevsky, Alex and Ozdaglar, Asuman and Tsitsiklis, John N.},
  title = {On Distributed Averaging Algorithms and Quantization Effects},
  journal = {IEEE Transactions on Automatic Control},
  volume = {54}, number = {11}, pages = {2506--2517},
  year = {2009}, doi = {10.1109/TAC.2009.2031203}
}

@article{Haggstrom2012,
  author = {H{\"a}ggstr{\"o}m, Olle},
  title = {A Pairwise Averaging Procedure with Application to Consensus Formation in the {Deffuant} Model},
  journal = {Acta Applicandae Mathematicae},
  volume = {119}, number = {1}, pages = {185--201},
  year = {2012}, doi = {10.1007/s10440-011-9668-9}
}

@article{HSF2017,
  author = {Hay, M. J. and Schiff, J. and Fisch, N. J.},
  title = {On Extreme Points of the Diffusion Polytope},
  journal = {Physica A: Statistical Mechanics and its Applications},
  volume = {473}, pages = {225--236},
  year = {2017}, doi = {10.1016/j.physa.2017.01.038}
}

@article{AL2012,
  author = {Aldous, David and Lanoue, Daniel},
  title = {A Lecture on the Averaging Process},
  journal = {Probability Surveys},
  volume = {9}, pages = {90--102},
  year = {2012}, doi = {10.1214/11-PS184}
}

@article{CDSZ2022,
  author = {Chatterjee, Sourav and Diaconis, Persi and Sly, Allan and Zhang, Lingfu},
  title = {A Phase Transition for Repeated Averages},
  journal = {The Annals of Probability},
  volume = {50}, number = {1}, pages = {1--17},
  year = {2022}, doi = {10.1214/21-AOP1526}
}

\bigskip
{\small
\noindent\textit{\textsuperscript{1}Western Institute of Maternal and Child Medicine, West China Second University Hospital, Sichuan University
Chengdu 610041, Sichuan, China}\par
\noindent\textit{\textsuperscript{2}Shanghai Institute for Mathematics and Interdisciplinary Sciences, Shanghai 200433, China.}\par
}

\end{document}